\documentclass[10pt,oneside,english,reqno]{amsart}
\usepackage[T1]{fontenc}
\usepackage[utf8]{inputenc}
\usepackage[letterpaper]{geometry}
\usepackage{color}
\usepackage{babel}
\usepackage{verbatim}
\usepackage{mathtools}
\usepackage{amstext}
\usepackage{amsthm}
\usepackage[unicode=true,pdfusetitle,
 bookmarks=true,bookmarksnumbered=false,bookmarksopen=false,
 breaklinks=false,pdfborder={0 0 0},backref=false,colorlinks=true]
 {hyperref}

\makeatletter
\numberwithin{equation}{section}
\numberwithin{figure}{section}
\theoremstyle{plain}
\newtheorem{thm}{\protect\theoremname}
  \theoremstyle{definition}
  \newtheorem{defn}[thm]{\protect\definitionname}
  \theoremstyle{plain}
  \newtheorem{prop}[thm]{\protect\propositionname}
  \theoremstyle{plain}
  \newtheorem{cor}[thm]{\protect\corollaryname}
  \theoremstyle{plain}
  \newtheorem{lem}[thm]{\protect\lemmaname}

\usepackage{url}
\usepackage{pgf}
\allowdisplaybreaks

\DeclareMathOperator*{\argmax}{arg\,max}
\DeclareMathOperator*{\argmin}{arg\,min}
\usepackage{chngcntr}
\counterwithout{equation}{section}
\counterwithout{figure}{section}

\makeatother

\usepackage{listings}
  \providecommand{\corollaryname}{Corollary}
  \providecommand{\definitionname}{Definition}
  \providecommand{\lemmaname}{Lemma}
  \providecommand{\propositionname}{Proposition}
\providecommand{\theoremname}{Theorem}

\begin{document}

\title{Expected Regularized Total Variation of Brownian Motion}

\author{Alexander Dunlap}

\address{Department of Mathematics, Stanford University, Stanford, CA 94305}

\email{ajdunl2@math.stanford.edu}

\date{\today}
\begin{abstract}
We introduce a notion of regularized total variation on an interval
for continuous functions with unbounded variation. The definition
of regularized total variation is obtained from that of total variation
by subtracting a penalty for the size of the partition used to estimate
the variation. We present an explicit construction of a partition
achieving the regularized total variation, and use this construction
to estimate the expected regularized total variation of Brownian motion
on an interval.
\end{abstract}

\maketitle

\section{Introduction}

For a continuous function $f:[a,b]\to\mathbf{R}$, it is standard
to define the total variation
\begin{equation}
TV_{[a,b]}(f)=\sup_{k\ge0}TV_{[a,b],k}(f),\label{eq:tv}
\end{equation}
where
\[
TV_{[a,b],k}(f)=\max_{a=t_{0}<t_{1}<\cdots<t_{k+1}=b}\sum_{i=1}^{k+1}|f(t_{i})-f(t_{i-1})|.
\]
Note that $TV_{[a,b],k}(f)$ is an increasing function of $k$, so
we actually have
\[
TV_{[a,b]}(f)=\lim_{k\to\infty}TV_{[a,b],k}(f).
\]
If $TV(f)<\infty$, then one says that $f$ is of \emph{bounded variation},
which is to say that one cannot make $TV_{[a,b],k}(f)$ arbitrarily
large by increasing $k$. Conversely, if $f$ has unbounded variation,
then $TV_{[a,b],k}(f)$ grows arbitrarily large as $k\to\infty$.
To get a finer measurement of the oscillations of a function of unbounded
variation, then, we can penalize the growth of $k$ in the optimization
equation (\ref{eq:tv}). This leads us to define, for each $\lambda>0$,
the \emph{regularized total variation} 
\begin{equation}
\Phi_{[a,b],\lambda}(f)=\sup_{k\ge0}(TV_{[a,b],k}(f)-\lambda k)=\sup_{k\ge0}\left(\max_{a=t_{0}<t_{1}<\cdots<t_{k+1}=b}\sum_{i=1}^{k+1}|f(t_{i})-f(t_{i-1})|-\lambda k\right).\label{eq:phi}
\end{equation}
We use the term ``regularized'' because, if we think of the optimization
equation (\ref{eq:tv}) as attempting to find the piecewise linear
approximation to $f$ that captures as much of the oscillation of
$f$ as possible (although the supremum may of course not be achieved),
then the penalty term in the regularized equation (\ref{eq:phi})
attempts to prevent ``overfitting'' $f$ by using too many piecewise-linear
segments.

Brownian motion provides a standard example of a function with (almost
surely) unbounded variation on a finite interval. The goal of the
present paper is to evaluate $\mathbf{E}\Phi_{[0,1],\lambda}(W)$,
where $\{W_{t}\}_{0\le t\le1}$ is a standard real-valued Brownian
motion. We will prove the following.
\begin{thm}
\label{thm:main-thm}For each $\lambda>0$, we have
\[
0\le\mathbf{E}\Phi_{[0,1],\lambda}(W)-\frac{1}{\lambda}\le\lambda,
\]
and thus in particular $\mathbf{E}\Phi_{[0,1],\lambda}(W)=1/\lambda+O(\lambda)$
as $\lambda\to0.$
\end{thm}
The proof will proceed in three steps. In Section~\ref{sec:char},
we give an explicit characterization of the optimal partition in (\ref{eq:phi}).
In Section~\ref{sec:est-small}, we use this characterization along
with martingale methods and Brownian scaling to evaluate the asymptotic
behavior of $\mathbf{E}\Phi_{[0,1],\lambda}(W)$ as $\lambda\to0$.
Finally, in Section~\ref{sec:error-anal}, we use the Markov property
of Brownian motion to establish the error bound stated in Theorem~\ref{thm:main-thm}.

We note in passing that the methods we use here, especially the characterization
of the optimal partition, are quite specific to our particular notion
of regularized total variation. In particular, one might wish to consider
the quantity
\[
\sup_{k\ge0}\left(\max_{a=t_{0}<t_{1}<\cdots<t_{k+1}=b}\sum_{i=1}^{k+1}|f(t_{i})-f(t_{i-1})|^{p}-\lambda k\right)
\]
for some power $p$, but our method of characterization of the optimal
partition fails for this quantity. One essential difficulty is that
the question of whether an interval $[t_{j-1},t_{j}]$ of a partition
should be ``split'' at $s,s'\in[t_{j-1},t_{j}]$ to improve the
objective function (\ref{eq:phi}) depends on the positions of $f(s)$
and $f(s')$ in the interval $f([t_{j-1},t_{j}])$, not just on their
difference $f(s')-f(s)$, as is the case when $p=1$ according to
Proposition~\ref{prop:partition-structure-2}(\ref{enu:nowiggle})
below.

\section{\label{sec:char}Characterizing the optimal partition}

We begin by giving a characterization of a partition of the interval
for which the outer supremum in (\ref{eq:tv}) is achieved. (We establish
the existence of such a partition in Corollary~\ref{cor:exists-opt}.)
We begin by introducing some notation and terminology.
\begin{defn}
For a partition $P=[a=t_{0}<t_{1}<\cdots<t_{k}<t_{k+1}=b]$ of $[a,b]$,
define $|P|=k$ and, for $\lambda>0$, let $\Phi_{I,\lambda,P}(f)=\sum\limits _{i=1}^{k+1}|f(t_{i})-f(t_{i-1})|-\lambda|P|$.
(Thus $\Phi_{I,\lambda}(f)=\max\limits _{P}\Phi_{I,\lambda,P}(f).$)
\end{defn}

\begin{defn}
Let $f\in C^{0}[a,b]$. Suppose $[x,y]\subset[a,b]$. We say that
$[x,y]$ is an \emph{$uptick$} (resp.\emph{~downtick)} for $f$
if $f(t_{i})>f(t_{i-1})$ (resp.~$f(t_{i})<f(t_{i-1})$), and a \emph{$\lambda$-uptick}
(resp.\emph{~$\lambda$-downtick}) for $f$ if $f(t_{i})\ge f(t_{i-1})+\lambda$
(resp.~$f(t_{i})\le f(t_{i-1})-\lambda$).
\end{defn}
Our first proposition is based on observations to the effect that
a partition with certain properties cannot be optimal since moving
or removing one or more of its points would increase $\Phi_{I,\lambda,P}(f)$.
In particular, this will allow us to derive the existence of an optimal
partition as a corollary.
\begin{prop}
\label{prop:partition-structure}Let $\lambda>0$ and $f\in C^{0}[a,b]$.
Suppose that $P=[a=t_{0}<t_{1}<\cdots<t_{k}<t_{k+1}=b]$ is a partition
of $I=[a,b]$ such that $\Phi_{I,\lambda,P}(f)\ge\Phi_{I,\lambda,Q}(f)$
whenever $|Q|\le|P|$. Then
\begin{enumerate}
\item \label{enu:updown}If $1\le j\le k$, then $f(t_{j})-f(t_{j-1})$
and $f(t_{j})-f(t_{j+1})$ have the same nonzero sign. In other words,
exactly one of $[t_{j-1},t_{j}]$ and $[t_{j},t_{j+1}]$ is an uptick
and the other is a downtick.
\item \label{enu:minmax}

\begin{enumerate}
\item \label{enu:minmax-rep} If $a<t_{j-1}$ and $[t_{j-1},t_{j}]$ is
an uptick (resp.~downtick), then $f$ attains its minimum (resp.~maximum)
on $[t_{j-1},t_{j}]$ at $t_{j-1}$.
\item If $t_{j}<b$ and $[t_{j-1},t_{j}]$ is an uptick (resp.~downtick),
then $f$ attains its maximum (resp.~minimum) on $[t_{j-1},t_{j}]$
at $t_{j}$.
\end{enumerate}
\item \label{enu:localminmax}For $1\le j\le k$, if $[t_{j-1},t_{j}]$
is an uptick (resp.~downtick), then $f$ attains its maximum (resp.~minimum)
on $[t_{j-1},t_{j+1}]$ at $t_{j}$.
\item \label{enu:bigenough}If $a<t_{j-1}$ and $t_{j}<b$, then we have
$|f(t_{j})-f(t_{j-1})|\ge\lambda$.
\item \label{enu:bigenoughends} If $k\ge1$, then for each $1\le j\le k+1$,
we have $|f(t_{j})-f(t_{j-1})|\ge\lambda/2$.
\end{enumerate}
\end{prop}
\begin{proof}
We prove each part in turn.
\begin{enumerate}
\item Let $Q=[a=t_{0}<\cdots<t_{j-1}<t_{j+1}<\cdots<t_{k+1}=b]$, so $|Q|=k-1$
and thus we have 
\begin{align*}
0\le\Phi_{I,\lambda,P}(f)-\Phi_{I,\lambda,Q}(f) & =|f(t_{j+1})-f(t_{j-1})|-\left[|f(t_{j+1})-f(t_{j})|+|f(t_{j})-f(t_{j-1})|\right]-\lambda\\
 & <|f(t_{j+1})-f(t_{j-1})|-\left[|f(t_{j+1})-f(t_{j})|+|f(t_{j})-f(t_{j-1})|\right],
\end{align*}
so $|f(t_{j-1})-f(t_{j})|+|f(t_{j})-f(t_{j-1})|<|f(t_{j+1})-f(t_{j-1})|$.
But if $f(t_{j-1})-f(t_{j})$ and $f(t_{j})-f(t_{j-1})$ were both
nonpositive or nonnegative then equality would hold.
\item We prove \ref{enu:minmax-rep} in the case when $[t_{j-1},t_{j}]$
is an uptick. Fix $x\in(t_{j-1},t_{j})$ and let $Q=[a=t_{0}<\cdots<t_{j-2}<x<t_{j}<\cdots<t_{k+1}=b]$,
so $|Q|=k$ and thus we have (using part \ref{enu:updown})
\begin{align*}
0\le\Phi_{I,\lambda,P}(f)-\Phi_{I,\lambda,Q}(f) & =f(t_{j})-f(t_{j-1})+f(t_{j-2})-f(t_{j-1})-\left[|f(t_{j})-f(x)|+|f(x)-f(t_{j-2})|\right]\\
 & \le f(t_{j})-2f(t_{j-1})-f(t_{j-2})-[f(t_{j})-2f(x)+f(t_{j-2})]\\
 & =2[f(x)-f(t_{j-1})],
\end{align*}
so $f(x)\ge f(t_{j-1})$.
\item This is an immediate consequence of the first two statements.
\item We prove the case when $[t_{j-1},t_{j}]$ is an uptick. Let $Q=[a=t_{0}<\cdots<t_{j-2}<t_{j+1}<\cdots<t_{k+1}=b]$,
so $|Q|=k-2$ and thus we have
\begin{align*}
0 & \le\Phi_{I,\lambda,P}(f)-\Phi_{I,\lambda,Q}(f)\\
 & =f(t_{j})-f(t_{j+1})+f(t_{j})-f(t_{j-1})+f(t_{j-2})-f(t_{j-1})-|f(t_{j+1})-f(t_{j-2})|-2\lambda\\
 & =2[f(t_{j})-f(t_{j-1})]-(f(t_{j+1})-f(t_{j-2})+|f(t_{j+1})-f(t_{j-2})|)-2\lambda\\
 & =2[f(t_{j})-f(t_{j-1})]-2\max\{f(t_{j+1})-f(t_{j-2}),0\}-2\lambda\\
 & \le2[f(t_{j})-f(t_{j-1})-\lambda],
\end{align*}
so $f(t_{j})-f(t_{j-1})\ge\lambda$.
\item Given part \ref{enu:bigenough}, it is sufficient to prove the cases
$j=1$ and $j=k+1$. We will prove the case $j=1$; the case $j=k+1$
is the same. Suppose wlog that $f(t_{1})>f(t_{0}),f(t_{2}).$ Let
$Q=[a=t_{0}<t_{2}<\cdots<t_{k+1}=b]$, so $|Q|=k-1$ and
\begin{align*}
0\le\Phi_{I,\lambda,P}(f)-\Phi_{I,\lambda,Q}(f) & =f(t_{1})-f(t_{0})+f(t_{1})-f(t_{2})-|f(t_{0})-f(t_{2})|-\lambda\\
 & =2f(t_{1})-2\max\{f(t_{0}),f(t_{2})\}-\lambda\\
 & \le2[f(t_{1})-f(t_{0})]-\lambda,
\end{align*}
so $|f(t_{1})-f(t_{0})|=f(t_{1})-f(t_{0})\ge\lambda/2$.\qedhere
\end{enumerate}
\end{proof}
\begin{cor}
\label{cor:exists-opt}Let $\lambda>0$ and $f\in C^{0}[a,b]$. Then
there is a $K\ge0$ such that if $|P|>K$, then there is a partition
$Q$ with $|Q|\le|P|$ and $\Phi_{I,\lambda,Q}(f)>\Phi_{I,\lambda,P}(f)$.
In particular, there is a partition $P$ (with $|P|\le K$) so that
$\Phi_{I,\lambda,P}(f)=\Phi_{I,\lambda}(f)$.\end{cor}
\begin{proof}
Since $f$ is uniformly continuous on $[a,b]$, there is a $\delta>0$
so that if $|x-y|<\delta$, then $|f(x)-f(y)|<\lambda/2$. Let $K=(b-a)/\delta$.
If $P=[a=t_{0}<t_{1}<\cdots<t_{k+1}=b]$ satisfies $|P|>K$, then
by the pigeonhole principle there is a $j$ so that $|t_{j}-t_{j-1}|<\delta$
and hence $|f(t_{j})-f(t_{j-1})|<\lambda/2$, so by statement~\ref{enu:bigenoughends}
of Proposition~\ref{prop:partition-structure} there is a $Q$ with
$|Q|\le|P|$ and $\Phi_{I,\lambda,Q}(f)>\Phi_{I,\lambda,P}(f)$.

Let $\tilde{\mathcal{P}}_{k}$ be the space of partitions of size
$k$ with possibly-coincident points, which is to say partitions of
the form $[a=t_{0}\le t_{1}\le\cdots\le t_{k}\le t_{k+1}=b]$. If
we equip $\tilde{\mathcal{P}}_{k}$ with the usual topology, then
$\tilde{\mathcal{P}}_{k}$ is compact for each $k$ and $\Phi_{I,\lambda,P}(f)$
depends continuously on $P$. Therefore, the maximum of $\Phi_{I,\lambda,P}(f)$
on $\bigsqcup\limits _{k=1}^{r}\tilde{\mathcal{P}}_{k}$ is achieved
whenever $r<\infty$. But by the previous paragraph, the maximum cannot
be achieved at a $P$ with $|P|>K$, and since $\lambda>0$, the maximum
cannot be achieved at a partition with coincident points. Therefore,
the maximum of $\Phi_{I,\lambda,P}(f)$ over all partitions $P$ (without
coincident points) is achieved.
\end{proof}
Now that we know that an optimal partition exists, we impose further
conditions on such a partition, in addition to the ones we already
have according to Proposition~\ref{prop:partition-structure}.
\begin{prop}
\label{prop:partition-structure-2}Let $\lambda>0$ and $f\in C^{0}[a,b]$.
Suppose that $P=[a=t_{0}<t_{1}<\cdots<t_{k}<t_{k+1}=b]$ is a partition
of $I=[a,b]$ such that $\Phi_{I,\lambda,P}(f)=\Phi_{I,\lambda}(f),$
and moreover that $|P|$ is maximal among all such $P$. Then we have
\begin{enumerate}
\item \label{enu:minmaxends}

\begin{enumerate}
\item \label{enu:minmaxends-rep}If $[a,t_{1}]$ is an uptick (resp.~downtick)
then $\min\limits _{a\le x\le t_{1}}f(x)>f(a)-\lambda/2$ (resp.~$\max\limits _{a\le x\le t_{1}}f(x)<f(a)+\lambda/2$).
\item \label{enu:minmaxends-last}If $[t_{k},b]$ is an uptick (resp.~downtick)
then $\max\limits _{t_{k}\le x\le b}f(x)<f(b)+\lambda/2$ (resp.~$\min\limits _{t_{k}\le x\le b}f(x)>f(b)-\lambda/2$).
\end{enumerate}
\item \label{enu:nowiggle} If $[t_{j-1},t_{j}]$ is an uptick (resp.~downtick),
then $[t_{j-1},t_{j}]$ contains no $\lambda$-downtick (resp.~$\lambda$-uptick).
\end{enumerate}
\end{prop}
\begin{proof}
We prove each part in turn.
\begin{enumerate}
\item We prove \ref{enu:minmaxends-rep} in the case when $[t_{0},t_{1}]$
is an uptick. Fix $x\in(t_{0},t_{1})$ and let $Q=[a=t_{0}<x<t_{1}<\cdots<t_{k+1}=b]$,
so $|Q|=k+1$ and, by the maximality of $\Phi_{I,\lambda,P}(f)$ and
$|P|$, we have
\begin{align*}
0<\Phi_{I,\lambda,P}-\Phi_{I,\lambda,Q}(f) & =f(t_{1})-f(a)-\left[|f(t_{1})-f(x)|+|f(x)-f(a)|\right]+\lambda\\
 & \le f(t_{1})-f(a)-\left[f(t_{1})-f(x)+f(a)-f(x)\right]+\lambda\\
 & =2[f(x)-f(a)]+\lambda,
\end{align*}
so $f(x)>f(a)-\lambda/2$.
\item We prove the case when $[t_{j-1},t_{j}]$ is an uptick. Let $[x,y]\subset[t_{j-1},t_{j}]$
and let $Q=[a=t_{0}<\cdots<t_{j-1}<x<y<t_{j}<\cdots<t_{k+1}=b]$,
so $|Q|=k+2$ and, by the maximality of $\Phi_{I,\lambda,P}(f)$ and
$|P|$, we have
\begin{align*}
0<\Phi_{I,\lambda,P}(f)-\Phi_{I,\lambda,Q}(f) & =f(t_{j})-f(t_{j-1})-\left[|f(t_{j})-f(y)|+|f(y)-f(x)|+|f(x)-f(t_{j-1})|\right]+2\lambda\\
 & \le f(t_{j})-f(t_{j-1})-\left[f(t_{j})-f(y)+f(x)-f(y)+f(x)-f(t_{j-1})\right]+2\lambda\\
 & =2[f(y)-f(x)+\lambda],
\end{align*}
so $f(y)>f(x)-\lambda.$ This implies that $[t_{j-1},t_{j}]$ cannot
contain a $\lambda$-downtick.\qedhere
\end{enumerate}
\end{proof}
Having established all of the needed necessary conditions for a partition
to be optimal, we now set out to write down a construction of an optimal
partition. When $f$ is taken to be a stochastic process, our construction
will be in terms of stopping times for the natural filtration, which
will allow us to use martingale methods in our later analysis.
\begin{defn}
\label{def:taus}Fix $\lambda>0$ and let $[a,b]$ be an interval.
Let $f\in C^{0}[a,b]$ and extend $f$ to $[a,\infty)$ arbitrarily.
We define a sequence of values $\tau_{0}<\tau_{1}<\cdots\in[a,\infty]$
as follows. Let
\begin{align*}
\tau_{0}^{\text{up}} & =\min\left\{ t>a\,\middle|\,f(t)\ge f(a)+\lambda/2\right\} =\min\left\{ t>a\,\middle|\,\text{\ensuremath{[a,t]} is a \ensuremath{(\lambda/2)}-uptick}\right\} ,\\
\intertext{and}\tau_{0}^{\text{down}} & =\min\left\{ t>a\,\middle|\,f(t)\le f(a)-\lambda/2\right\} =\min\left\{ t>a\,\middle|\,\text{\ensuremath{[a,t]} is a \ensuremath{(\lambda/2)}-downtick}\right\} ,
\end{align*}
Let $\tau_{0}=\tau_{0}^{\text{up}}\wedge\tau_{0}^{\text{down}}$.
If $f(\tau_{0}\wedge b)>f(a)$, then we call $\tau_{0}$ an \emph{upstop},
while if $f(\tau_{0}\wedge b)<f(a)$, then we call $\tau_{0}$ a \emph{downstop.
}Now we inductively define $\tau_{j}$ for all $j\ge1$: if $\tau_{j-1}$
is an upstop then define
\begin{align*}
\tau_{j} & =\min\{t>\tau_{j-1}\mid\text{\ensuremath{[\tau_{j-1},t]} contains a \ensuremath{\lambda}-downtick}\}=\min\left\{ t>\tau_{j-1}\,\middle|\,\max\limits _{\tau_{j-1}\le s\le t}f(s)-f(t)\ge\lambda\right\} ,\\
\intertext{\text{and call \ensuremath{\tau_{j}} a downstop, while if \ensuremath{\tau_{j-1}} is a downstop then define}}\tau_{j} & =\min\{t>\tau_{j-1}\mid\text{\ensuremath{[\tau_{j-1},t]} contains a \ensuremath{\lambda}-uptick}\}=\min\left\{ t>\tau_{j-1}\,\middle|\,f(t)-\min\limits _{\tau_{j-1}\le s\le t}f(s)\ge\lambda\right\} .
\end{align*}
and call $\tau_{j}$ an upstop. (We adopt the usual convention that
$\min\emptyset=\infty$.)
\end{defn}

\begin{defn}
\label{def:mjs}With setup as in Definition~\ref{def:taus}, let
$m_{0}=f(a)$ and, for all $j\ge1$ such that $\tau_{j-1}<b$, define
\[
m_{j}=m_{j}=\begin{cases}
\max\limits _{\tau_{j-1}\le s\le\tau_{j}\wedge b}f(s) & \text{if \ensuremath{\tau_{j}} is a downstop, or}\\
\min\limits _{\tau_{j-1}\le s\le\tau_{j}\wedge b}f(s) & \text{if \ensuremath{\tau_{j}} is an upstop.}
\end{cases}
\]

\end{defn}
\begin{figure}
\begin{centering}
\input{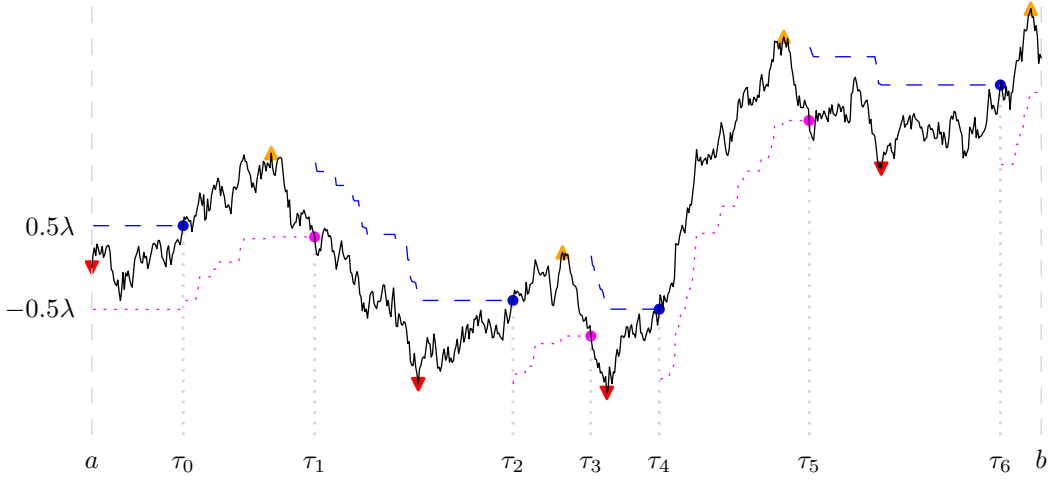}
\par\end{centering}

\caption{\label{fig:stops-illustration}Illustration of Definitions~\ref{def:taus}
and~\ref{def:mjs}. The function $f$ is drawn in solid black. The
blue dashed lines represent the height to which $f$ must climb in
order to establish the next upstop, while the magenta dotted lines
represent the depth to which $f$ must fall to establish the next
downstop. Thus, the upstops and downstops themselves are the times
at which the black line intersects the blue and magenta lines; these
points are marked by circles. Here, $\tau_{j}$ is an upstop if $j$
is even and a downstop if $j$ is odd. The $m_{j}$s are marked by
red and orange triangles.}
\end{figure}
See Figure~\ref{fig:stops-illustration} for a graphical interpretation
of Definitions~\ref{def:taus} and~\ref{def:mjs}.
\begin{prop}
\label{prop:mjtj}With the setup as in Definition~\ref{def:mjs},
if $j\ge1$ and $\tau_{j}<b$, then $m_{j}=f(\tau_{j})+\lambda$ if
$\tau_{j}$ is a downstop and $m_{j}=f(\tau_{j})-\lambda$ if $\tau_{j}$
is a downstop.\end{prop}
\begin{proof}
We prove the case in which $\tau_{j}$ is a downstop. In this case
we have
\[
\tau_{j}=\min\left\{ t>\tau_{j-1}\,\middle|\,\max\limits _{\tau_{j-1}\le s\le t}f(s)-f(t)\ge\lambda\right\} ,
\]
so since $f$ is continuous and $\tau_{j}<b$, we see that $m_{j}-f(\tau_{j})=\max\limits _{\tau_{j-1}\le s\le\tau_{j}}f(s)-f(\tau_{j})=\lambda$.\end{proof}
\begin{prop}
\label{prop:part-stoptimes}With setup as in Definition~\ref{def:mjs},
suppose that $P=[a=t_{0}<t_{1}<\cdots<t_{k}<t_{k+1}=b]$ is a partition
of $I=[a,b]$ such that $\Phi_{I,\lambda,P}(f)=\Phi_{I,\lambda}(f)$,
and moreover that $|P|$ is maximal over all such $P$. For each $j\ge1$
such that $\tau_{j-1}<b$, we have the following trichotomy.
\begin{enumerate}
\item \label{enu:genericcase}If $\tau_{j}<b$, then $k\ge j$ and $t_{j}\in[\tau_{j-1},\tau_{j}]$,
and $[t_{j-1},t_{j}]$ is an uptick or downtick as $\tau_{j-1}$ is
an upstop or downstop, respectively.
\item \label{enu:onemore}If $\tau_{j}\ge b$ and $|m_{j}-f(b)|\ge\lambda/2$,
then $k=j$ and $t_{j}\in[\tau_{j-1},b)$.
\item \label{enu:nomore}If $\tau_{j}\ge b$ and $|m_{j}-f(b)|<\lambda/2$,
then $k=j-1$.
\end{enumerate}
Moreover, whenever $1\le j\le k$, we have $f(t_{j})=m_{j}$.\end{prop}
\begin{proof}
We prove case~\ref{enu:genericcase} by induction on $j$.

We first prove the base case, $j=1$. Assume wlog that $\tau_{0}$
is an upstop. Proposition~\ref{prop:partition-structure}(\ref{enu:bigenoughends})
implies that $t_{1}\ge\tau_{0}$. Since we assume $\tau_{0}<b$, we
have $f(\tau_{0})=f(a)+\lambda/2$. Proposition~\ref{prop:partition-structure-2}(\ref{enu:minmaxends-rep})
implies that $[t_{0},t_{1}]$ cannot be a downtick, so it is an uptick.
Proposition~\ref{prop:partition-structure-2}(\ref{enu:nowiggle})
then implies that $[t_{0},t_{1}]$ cannot contain a $\lambda$-downtick,
but (since $\tau_{1}<b<\infty$) $[\tau_{0},\tau_{1}]$ contains a
$\lambda$-downtick, so $[t_{0},t_{1}]\not\supset[\tau_{0},\tau_{1}]$,
so $k\ge1$ and $t_{1}<\tau_{1}$. This proves the base case.

We now consider the inductive step, so assume that $t_{j-2}\le\tau_{j-2}\le t_{j-1}\le\tau_{j-1}<\tau_{j}<b$
and that $[t_{j-2},t_{j-1}]$ is an uptick if and only if $\tau_{j-2}$
is an upstop. Assume wlog that $\tau_{j-1}$ is a downstop, so $\tau_{j-2}$
is an upstop and $[t_{j-2},t_{j-1}]$ is an uptick. Proposition~\ref{prop:partition-structure}(\ref{enu:updown})
implies that $[t_{j-1},t_{j}]$ is a downtick. Since $[\tau_{j-1},\tau_{j}]$
contains a $\lambda$-uptick, Proposition~\ref{prop:partition-structure-2}(\ref{enu:nowiggle})
implies that $[t_{j-1},t_{j}]\not\supset[\tau_{j-1},\tau_{j}]$, so
$k\ge j$ and $t_{j}<\tau_{j}<b$ since $t_{j-1}\le\tau_{j-1}$. Proposition~\ref{prop:partition-structure}(\ref{enu:bigenough})
then implies that $[t_{j-1},t_{j}]$ contains a $\lambda$-downtick,
so $t_{j}\ge\tau_{j-1}$ since $t_{j-1}\ge\tau_{j-2}$ and $\tau_{j-1}$
is the first time $t$ such that $[\tau_{j-2},t]$ contains a $\lambda$-downtick.
This completes the proof of the inductive step for case~\ref{enu:genericcase}.

We now consider cases~\ref{enu:onemore}~and~\ref{enu:nomore},
so assume that $\tau_{j-1}<b\le\tau_{j}$. By case~\ref{enu:genericcase},
we see that $k\ge j-1$ and $t_{j-1}\le\tau_{j-1}$. Assume wlog that
$\tau_{j-1}$ is a downstop. In the same way as above, Proposition~\ref{prop:partition-structure}(\ref{enu:bigenough},
\ref{enu:bigenoughends}) implies that $t_{j}\ge\tau_{j-1}$ if $t_{j}<b$,
and if $t_{j}=b$ then of course $t_{j}\ge\tau_{j-1}$ as well. Since
$\tau_{j}\ge b$ is the first upstop after $\tau_{j-1}$, we observe
that $[\tau_{j-1},b]$ does not contain a $\lambda$-uptick, so $k\le j$
since if $k\ge j+1$, then $[t_{j},t_{j+1}]$ would have to be a $\lambda$-uptick
by Proposition~\ref{prop:partition-structure}(\ref{enu:bigenough}).

Under our wlog assumption, we have $m_{j}=\min\limits _{\tau_{j-1}\le s\le b}f(s)$.
If $m_{j}\le f(b)-\lambda/2$ and $k=j-1$, then $[t_{k-1},t_{k}=b]$
is a downtick and we obtain a contradiction of Proposition~\ref{prop:partition-structure-2}(\ref{enu:minmaxends-last}).
Thus, in case~\ref{enu:onemore}, we have $k=j$, completing the
proof for case~\ref{enu:onemore}. On the other hand, if $m_{j}>f(b)-\lambda/2$
and $k=j$, then $[t_{j},t_{j+1}=b]$ is an uptick and $f(t_{j})>f(b)-\lambda/2$,
contradicting Proposition~\ref{prop:partition-structure}(\ref{enu:bigenoughends}).
This completes the proof for case~\ref{enu:nomore}.

Finally, we show that $f(t_{j})=m_{j}$ whenever $1\le j\le k$. Assume
wlog that $[t_{j-1},t_{j}]$ is an uptick. We have shown that $t_{j-1}\le\tau_{j-1}\le t_{j}\le\tau_{j}\wedge b\le t_{j+1}$,
so $f(t_{j})\le\max\limits _{\tau_{j-1}\le s\le\tau_{j}\wedge b}f(s)=m_{j}$,
while Proposition~\ref{prop:partition-structure}(\ref{enu:localminmax})
tells us that $f(t_{j})=\max\limits _{t_{j-1}\le s\le t_{j+1}}f(s)\ge\max\limits _{\tau_{j-1}\le s\le\tau_{j}\wedge b}f(s)=m_{j}.$
Therefore, $f(t_{j})=m_{j}$.
\end{proof}

\section{\label{sec:est-small}Estimating \texorpdfstring{$\Phi_{I,\lambda}(W)$}{Φ\_\{I,λ\}(W)}
for small \texorpdfstring{$\lambda$}{λ}}

Proposition~\ref{prop:part-stoptimes} lets us algorithmically construct
a partition $P$ maximizing $\Phi_{I,\lambda,P}(f)$ for any $f$.
More importantly for our purposes, however, it lets us estimate $\Phi_{I,\lambda,P}(f)$
in terms of the $\tau_{j}$s and $f(\tau_{j})$s, as the next proposition
shows.
\begin{prop}
\label{prop:explicit-formula}Let $\lambda>0$ and let $I=[a,b]$
be an interval. Let $f\in C^{0}[a,b]$, define $\tau_{j}$ as above,
and let $k'=0\vee\max\{j\mid\tau_{j}<b\}$. Then there exists an $\alpha\in\{0,1\}$,
depending only on the values of $f$ on $[a,\tau_{0}\wedge b]$, so
that 
\[
\Phi_{I,\lambda}(f)=\lambda k'+|f(\tau_{0}\wedge b)-f(a)|+\sum_{j\ge1}(-1)^{j+\alpha}(f(\tau_{j}\wedge b)-f(\tau_{j-1}\wedge b))+\max\{0,2|m_{k'+1}-f(b)|-\lambda\}.
\]
\end{prop}
\begin{proof}
Choose a partition $P=[a=t_{0}<t_{1}<\cdots<t_{k+1}=b]$ of $[a,b]$
so that $\Phi_{I,\lambda,P}(f)=\Phi_{I,\lambda}(f)$ and $|P|$ is
maximal among all $P$ with this property. Let $\mathbf{1}_{\ge}=\mathbf{1}_{|m_{k'+1}-f(b)|\ge\lambda/2}$
and $\mathbf{1}_{<}=\mathbf{1}_{|m_{k'+1}-f(b)|<\lambda/2}$. Note
that $k=k'+\mathbf{1}_{\ge}$. Applying Proposition~\ref{prop:part-stoptimes},
we can write
\begin{multline}
\Phi_{I,\lambda}(f)=\Phi_{I,\lambda,P}(f)=\sum\limits _{j=1}^{k+1}|f(t_{j})-f(t_{j-1})|-\lambda k=\sum_{j\ge1}|f(t_{j})-f(t_{j-1})|\mathbf{1}_{j\le k+1}-\lambda k\\
=|m_{1}-f(a)|\mathbf{1}_{k'\ge1}+\sum_{j=2}^{k'}|m_{j}-m_{j-1}|+(|m_{k'+1}-m_{k'}|+|f(b)-m_{k'+1}|)\mathbf{1}_{\ge}\\
+|f(b)-m_{k'}|\mathbf{1}_{<}-\lambda(k'+\mathbf{1}_{\ge}).\label{eq:bigphi}
\end{multline}
Let $\alpha=0$ if $[t_{0},t_{1}]$ is a downtick and $\alpha=1$
if $[t_{0},t_{1}]$ is an uptick, so $j+\alpha$ is even if $[t_{j-1},t_{j}]$
is an uptick and odd if $[t_{j-1},t_{j}]$ is a downtick. Note that
this definition of $\alpha$ only depends on the values of $f$ on
$[a,\tau_{0}\wedge b]$. For $0\le j\le k'$, we have that $j+\alpha$
is even if $\tau_{j}$ is a downstop and odd if $\tau_{j}$ is an
upstop, so, by Proposition~\ref{prop:mjtj}, we can write $m_{j}=f(\tau_{j})+(-1)^{j+\alpha}\lambda$
whenever $1\le j\le k'$.

Now we can simplify each of the pieces of (\ref{eq:bigphi}) in turn.
We have
\begin{multline}
|m_{1}-f(a)|\mathbf{1}_{\tau_{1}<b}=(-1)^{\alpha+1}(m_{1}-f(a))\mathbf{1}_{k'\ge1}=(-1)^{\alpha+1}(f(\tau_{1})+(-1)^{\alpha+1}\lambda-f(a))\mathbf{1}_{k'\ge1}\\
=[(-1)^{\alpha+1}(f(\tau_{1})-f(a))+\lambda]\mathbf{1}_{k'\ge1}.\label{eq:bigphi1}
\end{multline}

We also have
\begin{align*}
|f(b)-m_{k'}| & =(-1)^{k'+\alpha+1}(f(b)-m_{k'})\\
 & =(-1)^{k'+\alpha+1}\left(f(b)-(f(\tau_{k'})+(-1)^{k'+\alpha}\lambda)\mathbf{1}_{k'\ge1}-f(a)\mathbf{1}_{k'=0}\right)\\
 & =(-1)^{k'+\alpha+1}[(f(b)-f(\tau_{k'}))\mathbf{1}_{k'\ge1}+(f(b)-f(a))\mathbf{1}_{k'=0}]+\lambda\mathbf{1}_{k'\ge1}
\end{align*}

Moreover, it is not hard to see that either $m_{k'}\le f(b)\le m_{k'+1}$
or $m_{k'}\ge f(b)\ge m_{k'+1}$. (If $k'=0$ then this is true by
the definitions. If $k'>0$ and, for example, if $\tau_{k'}$ is an
upstop, then $m_{k'+1}=\max\limits _{\tau_{k'}\le s\le\tau_{k'+1}\wedge b}f(s)=\max\limits _{\tau_{k'}\le s\le b}f(s)\ge f(b)$,
and if $f(b)<m_{k'}$, then $f(b)<m_{k'}\le f(\tau_{k'})-\lambda$,
contradicting the definition of $k'$.) This implies
\begin{align*}
|m_{k'+1}-m_{k'}|+|f(b)-m_{k'+1}| & =2|m_{k'+1}-f(b)|+|f(b)-m_{k'}|.
\end{align*}
Therefore,
\begin{align}
(|m_{k'+1}-m_{k'}|+\mathrlap{|f(b)-m_{k'+1}|)\mathbf{1}_{\ge}+|f(b)-m_{k'}|\mathbf{1}_{<}}\nonumber \\
 & =(2|m_{k'+1}-f(b)|+|f(b)-m_{k'}|)\mathbf{1}_{\ge}+|f(b)-m_{k'}|\mathbf{1}_{<}\nonumber \\
 & =2|m_{k'+1}-f(b)|\mathbf{1}_{\ge}+(-1)^{k'+\alpha+1}[(f(b)-f(\tau_{k'}))\mathbf{1}_{k'\ge1}+(f(b)-f(a))\mathbf{1}_{k'=0}]+\lambda\mathbf{1}_{k'\ge1}.\label{eq:bigphi2}
\end{align}

Finally, we can write
\begin{multline}
\sum_{j=2}^{k'}|m_{j}-m_{j-1}|=\sum_{j=2}^{k'}(-1)^{j+\alpha}(m_{j}-m_{j-1})=\sum_{j=2}^{k'}(-1)^{j+\alpha}(f(\tau_{j})+(-1)^{j+\alpha}\lambda-(f(\tau_{j-1})-(-1)^{j+\alpha}\lambda))\\
=2\lambda(k'-1+\mathbf{1}_{k'=0})+\sum_{j\ge2}(-1)^{j+\alpha}(f(\tau_{j})-f(\tau_{j-1}))\mathbf{1}_{k'\ge j}.\label{eq:bigphi3}
\end{multline}

Substituting (\ref{eq:bigphi1}), (\ref{eq:bigphi2}), and (\ref{eq:bigphi3})
into (\ref{eq:bigphi}), we obtain
\begin{align*}
\Phi_{I,\lambda}(f) & =\left[\lambda+(-1)^{\alpha+1}(f(\tau_{1})-f(a))\right]\mathbf{1}_{k'\ge1}+2\lambda(k'-1+\mathbf{1}_{k'=0})+\sum_{j\ge2}(-1)^{j+\alpha}(f(\tau_{j})-f(\tau_{j-1}))\mathbf{1}_{k'\ge j}\\
 & \qquad+2|m_{k'+1}-f(b)|\mathbf{1}_{\ge}+(-1)^{k'+\alpha+1}[(f(b)-f(\tau_{k'}))\mathbf{1}_{k'\ge1}+(f(b)-f(a))\mathbf{1}_{k'=0}]+\lambda\mathbf{1}_{k'\ge1}-\lambda(k'+\mathbf{1}_{\ge})\\
 & =\lambda k'+(-1)^{\alpha+1}(f(\tau_{1})-f(a))\mathbf{1}_{k'\ge1}+\sum_{j\ge2}(-1)^{j+\alpha}(f(\tau_{j})-f(\tau_{j-1}))\mathbf{1}_{k'\ge j}\\
 & \qquad+(2|m_{k'+1}-f(b)|-\lambda)\mathbf{1}_{\ge}+(-1)^{k'+\alpha+1}[(f(b)-f(\tau_{k'}))\mathbf{1}_{k'\ge1}+(f(b)-f(a))\mathbf{1}_{k'=0}]\\
 & =\lambda k'+(-1)^{\alpha+1}(f(\tau_{1})-f(a))\mathbf{1}_{k'\ge1}+(-1)^{k'+\alpha+1}(f(b)-f(\tau_{k'}))\mathbf{1}_{k'\ge1}+(-1)^{\alpha+1}(f(b)-f(a))\mathbf{1}_{k'=0}\\
 & \qquad+\sum_{j\ge2}(-1)^{j+\alpha}(f(\tau_{j})-f(\tau_{j-1}))\mathbf{1}_{\tau_{j}<b}+\max\{0,2|m_{k'+1}-f(b)|-\lambda\}\\
 & =\lambda k'+(-1)^{\alpha+1}(f(\tau_{1}\wedge b)-f(a))+\sum_{j\ge2}(-1)^{j+\alpha}(f(\tau_{j}\wedge b)-f(\tau_{j-1}\wedge b))+\max\{0,2|m_{k'+1}-f(b)|-\lambda\}\\
 & =\lambda k'+(-1)^{\alpha+1}(f(\tau_{1}\wedge b)-f(\tau_{0}\wedge b)+f(\tau_{0}\wedge b)-f(a))\\
 & \qquad+\sum_{j\ge2}(-1)^{j+\alpha}(f(\tau_{j}\wedge b)-f(\tau_{j-1}\wedge b))+\max\{0,2|m_{k'+1}-f(b)|-\lambda\}\\
 & =\lambda k'+|f(\tau_{0}\wedge b)-f(a)|+\sum_{j\ge1}(-1)^{j+\alpha}(f(\tau_{j}\wedge b)-f(\tau_{j-1}\wedge b))+\max\{0,2|m_{k'+1}-f(b)|-\lambda\},
\end{align*}
as claimed.
\end{proof}

\begin{prop}
\label{prop:interarrival-analysis}Let $\lambda>0$ and $I=[a,b]$
be an interval. Let $\{W_{t}\}_{t\ge a}$ be a standard Brownian motion
defined on the time interval $[a,b]$. Let $\{\mathcal{F}_{t}\}_{t\ge a}$
be the natural filtration of $\{W_{t}\}$. Then $\tau_{0}<\tau_{1}<\cdots$,
defined as in Definition~\ref{def:taus} for $W$, are almost-surely
finite stopping times with respect to the filtration $\{\mathcal{F}_{t}\}_{t\ge a}$.
Moreover, $\{\tau_{j}-\tau_{j-1}\}_{j\ge1}$ is an iid collection
of random variables and $\mathbf{E}(\tau_{j}-\tau_{j-1})=\lambda^{2}$
for each $j\ge1$. Finally, $\mathbf{E}\tau_{0}=(\lambda/2)^{2}$.\end{prop}
\begin{proof}
The fact that each $\tau_{j}$ is a stopping time is clear from the
definition. That $\{\tau_{j}-\tau_{j-1}\}_{j\ge1}$ are iid follows
immediately from the definition, given the strong Markov property
of Brownian motion and the fact that the negative of a standard Brownian
motion is another standard Brownian motion.

For $j\ge1$, to prove that $\tau_{j}-\tau_{j-1}$ is almost-surely
finite and to compute $\mathbf{E}(\tau_{j}-\tau_{j-1})$, we note
that, by the strong Markov property and the fact that a negative of
a Brownian motion is another Brownian motion, $\tau_{j}-\tau_{j-1}$
has the same distribution as the stopping time $\rho=\min\{t>a\mid Y_{t}\ge\lambda\},$
where $Y_{t}=\max\limits _{a\le s\le t}W_{s}-W_{t}$. %
Using the reflection principle, it can be shown that the process $\{Y_{t}\}_{t\ge a}$
has the same finite-dimensional distributions as the process $\{|W_{t}|\}_{t\ge a}$~\cite[Problem 2.8.8 or Theorem 3.6.17]{karshr}.
Let $\tilde{\rho}=\min\{t>a\mid|W_{t}|\ge\lambda\}$. Since $\{Y_{t}\}_{t\ge a}$
and $\{W_{t}\}_{t\ge a}$ have continuous sample paths, $\rho$ and
$\tilde{\rho}$ are both measurable with respect to the $\sigma$-algebras
generated by finite projections of $\{Y_{t}\}_{t\ge a}$ and $\{W_{t}\}_{t\ge a}$,
respectively, and so the distributions of $\rho$ and $\tilde{\rho}$
are the same since the finite-dimensional distributions of $\{Y_{t}\}_{t\ge a}$
and $\{W_{t}\}_{t\ge a}$ are the same. The fact that $\tilde{\rho}$
is almost-surely finite is the standard fact that Brownian motion
is almost-surely unbounded, and the computation $\mathbf{E}\tilde{\rho}=\lambda^{2}$
is a standard application of the optional sampling theorem on the
martingale $W_{t}^{2}-t$. Therefore, $\rho$ is almost-surely finite
and $\mathbf{E}\rho=\lambda^{2}$. %
We obtain $\mathbf{E}\tau_{0}=(\lambda/2)^{2}$ in the same way. %
{} \end{proof}
\begin{lem}
\label{lem:ui}With setup as in Proposition~\ref{prop:interarrival-analysis},
the sequence
\[
\left\{ \sum_{j=1}^{N}(-1)^{j+\alpha}(W_{\tau_{j}\wedge b}-W_{\tau_{j-1}\wedge b})\right\} _{N\ge0}
\]
 is bounded in $L^{2}$ by $b$ (and hence is uniformly integrable).\end{lem}
\begin{proof}
For each $j\ge1$, we know that $\tau_{j-1}\wedge b$ is a stopping
time, so by the strong Markov property of Brownian motion, $\left\{ (W_{t+\tau_{j-1}\wedge b}-W_{\tau_{j-1}\wedge b})^{2}-t\right\} _{t\ge0}$
is a martingale with respect to the filtration $\{\mathcal{F}_{t+\tau_{j-1}\wedge b}\}_{t\ge0}$.
Since $\tau_{j}\wedge b-\tau_{j-1}\wedge b$ is a bounded stopping
time with respect to this filtration, we have that $\mathbf{E}(W_{\tau_{j}\wedge b}-W_{\tau_{j-1}\wedge b})^{2}=\mathbf{E}(\tau_{j}\wedge b-\tau_{j-1}\wedge b)$
by the optional stopping theorem. Moreover, if $j\ne j'$, then $W_{\tau_{j}\wedge b}-W_{\tau_{j-1}\wedge b}$
and $W_{\tau_{j'}\wedge b}-W_{\tau_{j'-1}\wedge b}$ have mean zero,
and are independent by the strong Markov property. So
\begin{multline*}
\mathbf{E}\left[\sum_{j=1}^{N}(-1)^{j+\alpha}(W_{\tau_{j}\wedge b}-W_{\tau_{j-1}\wedge b})\right]^{2}=\\
=\sum_{j=1}^{N}\mathbf{E}(W_{\tau_{j}\wedge b}-W_{\tau_{j-1}\wedge b})^{2}+\sum_{1\le j\ne j'\le N}(-1)^{j+j'}\mathbf{E}\left[(W_{\tau_{j}\wedge b}-W_{\tau_{j-1}\wedge b})(W_{\tau_{j'}\wedge b}-W_{\tau_{j'-1}\wedge b})\right]=\\
=\sum_{j=1}^{N}\mathbf{E}(\tau_{j}\wedge b-\tau_{j-1}\wedge b)=\mathbf{E}(\tau_{N}\wedge b-\tau_{0}\wedge b)\le b.\qedhere
\end{multline*}
\end{proof}
\begin{lem}
\label{lem:mg-vanish}With setup as in Proposition~\ref{prop:interarrival-analysis},
we have $\mathbf{E}\sum\limits _{j\ge1}(-1)^{j+\alpha}(W_{\tau_{j}\wedge b}-W_{\tau_{j-1}\wedge b})=0.$\end{lem}
\begin{proof}
Note that $\alpha\in\mathcal{F}_{\tau_{0}\wedge b}$ by Proposition~\ref{prop:explicit-formula}.
Moreover, by the strong Markov property, $W_{\tau_{j}\wedge b}-W_{\tau_{j-1}\wedge b}$
is independent of $\mathcal{F}_{\tau_{0}\wedge b}$ for each $j\ge1$.
So for each $j\ge1$ we have
\begin{align*}
\mathbf{E}\left(-1\right)^{j+\alpha}(W_{\tau_{j}\wedge b}-W_{\tau_{j-1}\wedge b}) & =\mathbf{E}\left(-1\right)^{j+\alpha}\mathbf{E}\left[W_{\tau_{j}\wedge b}-W_{\tau_{j-1}\wedge b}\right]=0.
\end{align*}
Combined with Lemma~\ref{lem:ui}, this implies that $\mathbf{E}\sum\limits _{j\ge1}(-1)^{j+\alpha}(W_{\tau_{j}\wedge b}-W_{\tau_{j-1}\wedge b})=0,$
as claimed.\end{proof}
\begin{prop}
If $\{W_{t}\}_{t\ge0}$ is a standard Brownian motion, then 
\[
\mathbf{E}\Phi_{[0,b],\lambda}(W)\sim\frac{b}{\lambda}
\]
as $b\to\infty$.\end{prop}
\begin{proof}
Let $I=[0,b]$, and set notation as in Proposition~\ref{prop:explicit-formula}
with $f(t)=W_{t}$. Proposition~\ref{prop:explicit-formula} then
tells us that
\begin{align*}
\mathbf{E}\Phi_{I,\lambda}(W) & =\lambda\mathbf{E}k'(b)+\mathbf{E}|W_{\tau_{0}\wedge b}|+\mathbf{E}\sum_{j\ge1}(-1)^{j+\alpha}(W_{\tau_{j}\wedge b}-W_{\tau_{j-1}\wedge b})+\mathbf{E}\max\{0,2|m_{k'(b)+1}-W_{b}|-\lambda\},
\end{align*}
where $k'(b)=0\vee\max\{j\mid\tau_{j}<b\}$. Elementary renewal theory
and Proposition~\ref{prop:interarrival-analysis} then let us write
\[
\lim_{b\to\infty}\frac{\mathbf{E}k'(b)}{b}=\frac{1}{\mathbf{E}(\tau_{j}-\tau_{j-1})}=\frac{1}{\lambda^{2}}.
\]
Moreover, Lemma~\ref{lem:mg-vanish} says that $\mathbf{E}\sum\limits _{j\ge1}(-1)^{j+\alpha}(W_{\tau_{j}\wedge b}-W_{\tau_{j-1}\wedge b})=0.$
We note that
\[
0\le\mathbf{E}|W_{\tau_{0}\wedge b}|\le\lambda/2,
\]
and
\[
|\mathbf{E}\max\{0,2|m_{k'(b)+1}-W_{b}|-\lambda\}|\le\mathbf{E}|\max\{0,2|m_{k'(b)+1}-W_{b}|-\lambda\}|<\lambda
\]
by the definition of $k'$. Therefore, we have
\[
\lim_{b\to\infty}\frac{1}{b}\mathbf{E}\Phi_{I,\lambda}(W)=\frac{\lambda}{\lambda^{2}}=\frac{1}{\lambda},
\]
as claimed.
\end{proof}

\begin{prop}
\label{prop:scaling}Let $b\ge0$ and $\mu>0$, and let $\{W_{t}\}_{t\ge0}$
be a standard Brownian motion. Then $\Phi_{[0,\mu b],\lambda}(W)$
has the same law as $\sqrt{\mu}\Phi_{[0,b],\lambda/\sqrt{\mu}}(W)$.
In particular, $\mathbf{E}\Phi_{[0,\mu b],\lambda}(W)=\sqrt{\mu}\mathbf{E}\Phi_{[0,b],\lambda/\sqrt{\mu}}(W).$\end{prop}
\begin{proof}
This follows from Brownian scaling by the simple computation
\begin{align*}
\Phi_{[0,\mu b],\lambda}(W) & =\max\limits _{k\ge0}\max\limits _{0=t_{0}<\cdots<t_{k+1}=\mu b}\left(\sum\limits _{i=1}^{k+1}|W_{t_{i}}-W_{t_{i-1}}|-\lambda k\right)\\
 & =\sqrt{\mu}\cdot\max\limits _{k\ge0}\max\limits _{0=s_{0}<\cdots<s_{k+1}=b}\left(\sum\limits _{i=1}^{k+1}\left|\frac{1}{\sqrt{\mu}}W_{\mu s_{i}}-\frac{1}{\sqrt{\mu}}W_{\mu s_{i-1}}\right|-\frac{\lambda}{\sqrt{\mu}}\cdot k\right)\\
 & =\sqrt{\mu}\Phi_{[0,b],\lambda/\sqrt{\mu}}(\tilde{W}),
\end{align*}
where $\tilde{W}_{t}=W_{\mu t}/\sqrt{\mu}$. But by Brownian scaling,
$\tilde{W}_{t}$ is another standard Brownian motion, so $\Phi_{[0,\mu b],\lambda}(W)$
has the same law as $\sqrt{\mu}\Phi_{[0,b],\lambda/\sqrt{\mu}}(W)$.\end{proof}
\begin{cor}
\label{cor:limit}If $\{W_{t}\}_{t\ge0}$ is a standard Brownian motion,
then
\[
\mathbf{E}\Phi_{[0,1],\lambda}(W)\sim1/\lambda
\]
as $\lambda\downarrow0$.\end{cor}
\begin{proof}
An easy computation from the previous two propositions:
\[
\lim_{\lambda\downarrow0}\lambda\mathbf{E}\Phi_{[0,1],\lambda}(W)=\lim_{\lambda\downarrow0}\lambda^{2}\mathbf{E}\Phi_{[0,1/\lambda^{2}],1}(W)=\lim_{b\to\infty}\frac{1}{b}\mathbf{E}\Phi_{[0,b]}(W)=1.\qedhere
\]

\end{proof}

\section{\label{sec:error-anal}Error analysis for fixed \texorpdfstring{$\lambda$}{λ}}

Now that we have established the asymptotic behavior (as $\lambda\to0$)
of $\mathbf{E}\Phi_{[0,1],\lambda}(W)$, we turn to estimating the
error term. We use a simple argument based on subdividing intervals
and applying Brownian scaling.
\begin{prop}
\label{prop:crude-split}Let $f\in C^{0}[a,b]$. For any $a<b<c$,
we have $\Phi_{[a,b],\lambda}(f)+\Phi_{[b,c],\lambda}(f)-\lambda\le\Phi_{[a,c],\lambda}(f)\le\Phi_{[a,b],\lambda}(f)+\Phi_{[b,c],\lambda}(f)$.\end{prop}
\begin{proof}
We compute 
\begin{align*}
\Phi_{[a,c],\lambda}(f) & =\max\limits _{k\ge0}\max\limits _{a=t_{0}<\cdots<t_{k+1}=c}\left(\sum\limits _{i=1}^{k+1}|f(t_{i})-f(t_{i-1})|-\lambda k\right)\\
 & \ge\max\limits _{k\ge1}\max\limits _{1\le\ell\le k}\max\limits _{\substack{a=t_{0}<\cdots<t_{\ell}=b\\
b=t_{\ell}<\cdots<t_{k+1}=c
}
}\left(\sum\limits _{i=1}^{k+1}|f(t_{i})-f(t_{i-1})|-\lambda k\right)\\
 & =\max\limits _{k\ge1}\max\limits _{1\le\ell\le k}\max\limits _{\substack{a=t_{0}<\cdots<t_{\ell}=b\\
b=t_{\ell}<t_{\ell+2}<\cdots<t_{k+1}=c
}
}\left(\sum\limits _{i=1}^{\ell}|f(t_{i})-f(t_{i-1})|-\lambda(\ell-1)-\lambda+\sum\limits _{i=\ell+1}^{k+1}|f(t_{i})-f(t_{i-1})|-\lambda(k-\ell)\right)\\
 & =\max\limits _{\ell'\ge0}\max\limits _{a=t_{0}<\cdots<t_{\ell'+1}=b}\left(\sum\limits _{i=1}^{\ell}|f(t_{i})-f(t_{i-1})|-\lambda(\ell-1)\right)\\
 & \qquad\qquad+\max\limits _{m\ge0}\max\limits _{b=t_{0}<\cdots<t_{m+1}=c}\left(\sum\limits _{i=0}^{m+1}|f(t_{i})-f(t_{i-1})|-\lambda m\right)-\lambda\\
 & =\Phi_{[a,b],\lambda}(f)+\Phi_{[b,c],\lambda}(f)-\lambda.
\end{align*}
Moreover, 
\begin{align*}
\mathrlap{\Phi_{[a,c],\lambda}(f)=}\quad\\
 & =\max\limits _{k\ge0}\max\limits _{a=t_{0}<\cdots<t_{k+1}=c}\left(\sum\limits _{i=1}^{k+1}|f(t_{i})-f(t_{i-1})|-\lambda k\right)\\
 & =\max_{\substack{k\ge0\\
1\le\ell\le k
}
}\max_{\substack{a=t_{0}<\cdots<t_{\ell}<b\\
b<t_{\ell+1}<\cdots<t_{k+1}=c
}
}\left(\sum\limits _{i=1}^{\ell}|f(t_{i})-f(t_{i-1})|+|f(b)-f(t_{\ell})|+|f(t_{\ell+1})-f(b)|+\sum\limits _{i=\ell+2}^{k}|f(t_{i})-f(t_{i-1})|-\lambda k\right)\\
 & \le\max_{\substack{k\ge0\\
1\le\ell\le k
}
}\max_{\substack{a=t_{0}<\cdots<t_{\ell}<b\\
b<t_{\ell+1}<\cdots<t_{k+1}=c
}
}\Bigg(\left[\sum\limits _{i=1}^{\ell}|f(t_{i})-f(t_{i-1})|+|f(b)-f(t_{\ell})|-\lambda\ell\right]\\
 & \qquad\qquad\qquad\qquad\qquad\qquad\qquad+\left[|f(t_{\ell+1})-f(b)|+\sum\limits _{i=\ell+2}^{k+1}|f(t_{i})-f(t_{i-1})|-\lambda(k-\ell)\right]\Bigg)\\
 & =\Phi_{[a,b],\lambda}(f)+\Phi_{[b,c],\lambda}(f).
\end{align*}
In summary, we have $\Phi_{[a,b],\lambda}(f)+\Phi_{[b,c],\lambda}(f)-\lambda\le\Phi_{[a,c],\lambda}(f)\le\Phi_{[a,b],\lambda}(f)+\Phi_{[b,c],\lambda}(f)$.\end{proof}
\begin{prop}
\label{prop:splitscale}If $\{W_{t}\}_{t\ge0}$ is a standard Brownian
motion, then for any $L\in\mathbf{N}$ we have 
\[
\mathbf{E}\Phi_{[0,b],\lambda/\sqrt{L}}(W)=\sqrt{L}(\mathbf{E}\Phi_{[0,b],\lambda}(W)-\varepsilon_{\lambda,L}),
\]
with $0\le\varepsilon_{\lambda,L}\le\lambda$.\end{prop}
\begin{proof}
By inductively applying Proposition~\ref{prop:crude-split} and the
Markov property of Brownian motion, for any $L\in\mathbf{N}$ we get
the inequality 
\[
L(\mathbf{E}\Phi_{[0,b],\lambda}(W)-\lambda)\le L\mathbf{E}\Phi_{[0,b],\lambda}(W)-(L-1)\lambda\le\mathbf{E}\Phi_{[0,Lb],\lambda}(W)\le L\mathbf{E}\Phi_{[0,b],\lambda}(W),
\]
so we have
\[
\mathbf{E}\Phi_{[0,Lb],\lambda}(W)=L(\mathbf{E}\Phi_{[0,b],\lambda}(W)-\varepsilon_{\lambda,L}),
\]
for some $0\le\varepsilon_{\lambda,L}\le\lambda$. By Proposition~\ref{prop:scaling},
we have $\mathbf{E}\Phi_{[0,Lb]}(W)=\sqrt{L}\mathbf{E}\Phi_{[0,b],\lambda/\sqrt{L}}(W),$
so the result follows.
\end{proof}

Our final proposition implies Theorem~\ref{thm:main-thm}.
\begin{prop}
With notation as in Proposition~\ref{prop:splitscale}, we have 
\[
\mathbf{E}\Phi_{[0,1],\lambda}(W)=\frac{1}{\lambda}+\alpha_{\lambda},
\]
where $\alpha_{\lambda}=\lim\limits _{r\to\infty}\varepsilon_{\lambda,2^{r}}$
(and thus $0\le\alpha_{\lambda}\le\lambda$).\end{prop}
\begin{proof}
For ease of notation, we now put $\xi(\lambda)=\mathbf{E}\Phi_{[0,1],\lambda}(f)$,
where $f(t)=W_{t}$, a standard Brownian motion. Note that $\xi$
is decreasing. With this notation, and putting $b=1$, the previous
Proposition tells us that 
\begin{equation}
\xi(\lambda/\sqrt{L})=\sqrt{L}\xi(\lambda)-\sqrt{L}\varepsilon_{\lambda,L}.\label{eq:recur}
\end{equation}
 %

For typographical convenience, put $\zeta=\sqrt{2}$. Applying (\ref{eq:recur})
three times, with $L=2^{r+1}$, $L=2$, and $L=2^{r}$, we obtain
\begin{multline*}
\zeta^{r+1}\xi(\lambda)-\zeta^{r+1}\varepsilon_{\lambda,2^{r+1}}=\xi(\lambda/\zeta^{r+1})=\zeta\xi(\lambda/\zeta^{r})-\zeta\varepsilon_{\lambda/\zeta^{r},2}=\zeta\left[\zeta^{r}\xi(\lambda)-\zeta^{r}\varepsilon_{\lambda,2{}^{r}}\right]-\zeta\varepsilon_{\lambda/\zeta^{r},2}\\
=\zeta^{r+1}\xi(\lambda)-\zeta^{r+1}\varepsilon_{\lambda,2^{r}}-\zeta\varepsilon_{\lambda/\zeta^{r},2},
\end{multline*}
so
\[
\zeta^{r}\varepsilon_{\lambda,2^{r+1}}=\zeta^{r}\varepsilon_{\lambda,2^{r}}+\varepsilon_{\lambda/\zeta^{r},2},
\]
so in particular 
\[
\varepsilon_{\lambda,2^{r+1}}\ge\varepsilon_{\lambda,2^{r}}
\]
since $\varepsilon_{\lambda/\zeta^{r},2}\ge0$. Thus for fixed $\lambda$,
the sequence $\{\varepsilon_{\lambda,2^{r}}\}_{r}$ is nondecreasing
and bounded above by $\lambda$. Thus $\alpha_{\lambda}=\lim\limits _{r\to\infty}\varepsilon_{\lambda,2^{r}}$
exists and lies in $[0,\lambda]$. From (\ref{eq:recur}) again, we
have
\[
\xi(\lambda)=\varepsilon_{\lambda,2^{r}}+\frac{1}{\zeta^{r}}\xi(\lambda/\zeta^{r}),
\]
so we can conclude, using Corollary~\ref{cor:limit}, that 
\[
1=\lim\limits _{r\to\infty}\frac{\lambda}{\zeta^{r}}\xi(\lambda/\zeta^{r})=\lambda\xi(\lambda)-\lambda\alpha_{\lambda},
\]
or $\xi(\lambda)=1/\lambda+\alpha_{\lambda}$, as claimed.
\end{proof}

\section*{Acknowledgments}

The author wishes to thank Jian Ding for suggesting the problem and
for extensive advice and comments on the manuscript. Much of this
work was done while the author was an undergraduate student at the
University of Chicago. Final preparation of the manuscript was done
while the author was supported by an NSF Graduate Research Fellowship.

\bibliographystyle{plain}
\bibliography{optimizeBM}

\end{document}